\documentclass[12pt]{amsart}
\usepackage{amssymb,amsmath,amscd}
\usepackage{latexsym,bm,bbm,mathrsfs}
\usepackage{hyperref,graphicx}
\usepackage{mathtools}
\usepackage{stmaryrd}
\usepackage[all,pdf]{xy}
\usepackage{graphicx}
\graphicspath{ {./} }

\newtheorem{thm}{Theorem}[section]

\newtheorem{example}[thm]{Example}
\newtheorem{theorem}[thm]{Theorem}

\newtheorem{prop}[thm]{Proposition}
\newtheorem{lemma}[thm]{Lemma}

\newtheorem{conj}[thm]{Conjecture}

\newtheorem{remark}[thm]{Remark}

\newtheorem{defn}[thm]{Definition}

\newcommand{\bba}{{\mathbb{A}}}

\newcommand{\bbz}{{\mathbb{Z}}}
\newcommand{\bbc}{\mathbb{C}}
\newcommand{\bbf}{\mathbb{F}}

\title[Punctual Hilbert schemes of points of $\bba^{3}$]{Punctual Hilbert schemes of points of $\bba^{3}$ in the Grothendieck group of varieties}

\author{Sailun Zhan}
\address{Department of Mathematical Sciences, Binghamton University, Binghamton, NY, 13902, U.S.A.}
\email{zhans@binghamton.edu}
\begin{document}

\begin{abstract}
We give an explicit stratification of the punctual Hilbert schemes of $n$ points of $\bba^{m+1}$ with respect to $m$-dimensional partitions in the Grothendieck group of varieties. As an application, we calculate the classes of the punctual Hilbert schemes of $n$ points of $\bba^3$ and the classes of the Hilbert schemes of $n$ points of $\bba^3$ in the Grothendieck of varieties for $n\leq 5$. 
\end{abstract}

\maketitle

\section{Introduction}

Let $k$ be a field, and let $n$ be a positive integer. Denote by ${\rm Hilb}_{0}^{n}(\bba^{m+1})$ the punctual Hilbert scheme of $n$ points of $\bba^{m+1}$, which parametrizes closed subschemes of length $n$ of $\bba^{m+1}$ which support on the origin. In other words, it parametrizes the codimension $n$ ideals of $k[[x_{0},x_{1},...,x_{m}]]$. Let $m=1$. When $k=\mathbb{C}$, ${\rm Hilb}_{0}^{n}(\bba^{2})$ has a cellular decomposition parametrized by the partitions of $n$ by the Bia\l ynicki-Birula theorem \cite{ES87}. 

\begin{defn}
Let $k$ be any field. The Grothendieck group of $k$-varieties $K_{0}({\rm Var}_{k})$ is the free abelian group generated by isomorphism classes of $k$-varieties modulo the relations $[X]=[Y]+[X\slash Y]$ for all pairs $(X,Y)$ consisting of a variety $X$ and a closed subvariety $Y$. By variety we mean a reduced separated scheme of finite type over $k$. It has a commutative ring structure by $[X][Y]=[X\times Y]$.
\end{defn}

When $k$ is any field, the decomposition of ${\rm Hilb}_{0}^{n}(\bba^{2})$ into affine spaces is also true and an explicit ``cell decomposition" and explicit parametrizations of the cells are given: 

\begin{theorem}\label{LL}\cite[Proposition A.2]{LL20}
Let $k$ be any field, and let $n$ be any positive integer. Then in the Grothendieck group of varieties $K_{0}({\rm Var}_{k})$,
\[
[{\rm Hilb}_{0}^{n}(\bba^{2})]=\sum_{\beta\in P(n)}[\bba^{n-|\beta|}],
\]
where $P(n)$ denotes the set of partitions $\beta$ of $n$, and $|\beta|$ is the number of parts of $\beta$.
\end{theorem}

Let $I\subset k[[x,y]]$ be a codimension $n$ ideal. Define $I_{k}:=(I:x^{k})$. Then $(I_{k}+(x))/(x)=(y^{\lambda_{k}})\subseteq k[y]$ for some $\lambda_{k}$, and $\lambda_{0}\geq\lambda_{1}\geq...\geq\lambda_{r}>\lambda_{r+1}=0$ gives a partition $\lambda$ of $n$. We say the ideal $I$ is of type $\lambda$. Let $\beta_{0}\geq\beta_{1}\geq...\geq\beta_{s}>\beta_{s+1}=0$ be the dual partition $\beta$. Then there is a one-to-one correspondence between the codimension $n$ ideals of type $\lambda$ in $k[[x,y]]$ and the $k$-valued points in $\bba^{n-|\beta|}$. We generalize this idea in \cite{LL20} to higher dimensions. 

Given two $n$-tuples $(a_{1},...,a_{n}),(b_{1},...,b_{n})\in\mathbb{N}^{n}$, we write $(a_{1},...,a_{n})\leq(b_{1},...,b_{n})$ if $a_{i}\leq b_{i}$ for all $1\leq i\leq n$. 
\begin{defn}
Let $d\geq 1$ and $n\geq 0$. A \emph{d-dimensional partition} $\lambda$ of $n$ is an array
\[
\lambda=(\lambda_{r_{1},...,r_{d}})_{r_{1},...,r_{d}}
\]
of nonnegative integers $\lambda_{r_{1},...,r_{d}}$ indexed by the tuples $(r_{1},...,r_{d})\in\mathbb{N}^{d}$ such that
\[
\sum_{r_{1},...,r_{d}}\lambda_{r_{1},...,r_{d}}=n,
\]
and $\lambda_{r_{1},...,r_{d}}\geq\lambda_{s_{1},...,s_{d}}$ if $(r_{1},...,r_{d})\leq(s_{1},...,s_{d})$. Denote by $P_{d}(n)$ the number of $d$-dimensional partitions of $n$.
\end{defn}

\begin{theorem}\label{decomp}
Given positive integers $n$ and $m$, we have the following decomposition for the punctual Hilbert scheme of $n$ points on $\bba^{m+1}$ in the Grothendieck group of varieties over a field $k$:
\[
Hilb_{0}^{n}(\bba^{m+1})=\sum_{\lambda}V_{\lambda},
\]
where $\lambda$ goes through all the $m$-dimensional partition of $n$, and $V_{\lambda}$ is an affine variety with explicit affine coordinates and relations.
\end{theorem}

Notice that when $m=1$, this prove Theorem \ref{LL} with another explicit parametrizations of the cells. In general, for example when $m=2$, $V_{\lambda}$ is not an affine space (see Example \ref{minimal}). However, it seems that they may be polynomials in $L$ in the Grothendieck group of varieties, where $L:=[\bba^1]$. So we make the following conjecture:

\begin{conj}
All the strata $V_{\lambda}$ in the puntual Hilbert schemes of points in $\bba^3$ are polynomials in $L$ in the Grothendieck group of varieties.
\end{conj}

Combining theoretical arguments and computer programme, we calculate the classes of the (punctual) Hilbert schemes of $n$ points of $\bba^3$ when $n\leq 5$. In principle, one can calculate the class for larger $n$.

\begin{thm}\label{genseries1}
The classes of the punctual Hilbert schemes of $n$ points of $\bba^3$ in $K_{0}({\rm Var}_k)$ for $n\leq 5$ are
\begin{center}
\begin{tabular}{c||c}
$n$ & $[{\rm Hilb}_{0}^{n}(\bba^{3})]$ \\
\hline\hline
$1$ & $1$ \\
$2$ & $L^2+L+1$ \\
$3$ & $L^4+L^3+2L^2+L+1$ \\
$4$ & $L^6+2L^5+3L^4+3L^3+2L^2+L+1$ \\
$5$ & $L^8+2L^7+4L^6+5L^5+5L^4+3L^3+2L^2+L+1$ \\
\end{tabular}
\end{center}
\end{thm}

\begin{thm}\label{genseries2}
The classes of the Hilbert schemes of $n$ points of $\bba^3$ in $K_{0}({\rm Var}_k)$ for $n\leq 9$ are
\begin{center}
\begin{tabular}{c||c}
$n$ & $[{\rm Hilb}^{n}(\bba^{3})]$ \\
\hline\hline
$1$ & $L^3$ \\
$2$ & $L^6+L^5+L^4$ \\
$3$ & $L^9+L^8+2L^7+L^6+L^5$ \\
$4$ & $L^{12} + L^{11}+3L^{10}+3L^9+4L^8+L^7+L^6-L^5$ \\
$5$ & $L^{15} + L^{14}+3L^{13}+4L^{12}+7L^{11}+5L^{10}+4L^{9}-L^6$ \\
\end{tabular}
\end{center}
\end{thm}

\begin{remark}
Notice that in the case of $[{\rm Hilb}_{0}^{n}(\bba^{3})]$, it seems that the coefficient of $L^{k}$ for a fixed $k$ becomes stable when $n$ is sufficiently large. This is the case for $[{\rm Hilb}_{0}^{n}(\bba^{2})]$, and the same phenomenon may hold for $[{\rm Hilb}_{0}^{n}(\bba^{3})]$ as well.
\end{remark}

\begin{center}
\sc{Acknowledgements}
\end{center}
\vspace{0.1 in}
I thank Michael Larsen for several helpful conversations. I thank Andrea Ricolfi, Joachim Jelisiejew, and Erik Nikolov for pointing out a mistake in an earlier draft.

\section{Punctual Hilbert schemes of points}

Let $I\subset k[[x_{0},x_{1},...,x_{m}]]$ be a codimension $n$ ideal. Notice that this is the same as giving a $(x_{0},...,x_{m})$-primary ideal $I\subset k[x_{0},...,x_{m}]$ of codimension $n$. Let $0< t< m$ be an integer. Define for $r_{1},...,r_{t}\geq 0$ inductively
\[
I_{r_1}:=(I:x_{1}^{r_1})/(x_1)\subset k[x_{0},x_{2},...,x_{m}],
\]
\[
I_{r_{1},r_{2},...,r_{t+1}}:=(I_{r_{1},r_{2},...,r_{t}}:x_{t+1}^{r_{t+1}})/(x_{t+1})\subset k[x_{0},x_{t+2},...,x_{m}],
\]
Then $I_{r_{1},...,r_{m}}=(x_{0}^{\lambda_{r_{1},...,r_{m}}})\subseteq k[x_{0}]$ for some $\lambda_{r_{1},...,r_{m}}\geq 0$.

\begin{prop}\label{partition}
The nonnegative integers $\lambda_{r_{1},...,r_{m}}$ give an $m$-dimensional partition $\lambda$ of $n$.
\end{prop}

\begin{proof}

Consider the short exact sequences for $0\leq t<m$
\[
0\to \frac{k[x_{0},x_{t+1},...,x_m]}{(I_{r_{1},...,r_{t}}:x_{t+1}^{r_{t+1}+1})}\xrightarrow{x_{t+1}}\frac{ k[x_{0},x_{t+1},...,x_m]}{(I_{r_{1},...,r_{t}}:x_{t+1}^{r_{t+1}})}\to \frac{k[x_{0},x_{t+2},...,x_m]}{I_{r_{1},...,r_{t+1}}}\to 0,
\]
where the second arrow is the multiplication by $x_{t+1}$ and the third arrow is the quotient map. Using the above short exact sequences and $I_{r_{1},...,r_{m}}=(x_{0}^{\lambda_{r_{1},...,r_{m}}})$, we deduce that
\[
\begin{aligned}
n=\dim k[x_{0},...,x_{m}]/I&=\sum_{r_{1}=0}^{\infty}\dim k[x_{0},x_2,...,x_{m}]/I_{r_{1}}\\
&=\sum_{r_{1}=0}^{\infty}\sum_{r_{2}=0}^{\infty}\dim k[x_{0},x_3,...,x_{m}]/I_{r_{1},r_{2}}\\
&=\cdots\\
&=\sum_{r_{1},...,r_{m}=0}^{\infty}\dim k[x_{0}]/I_{r_{1},...,r_{m}}\\
&=\sum_{r_{1},...,r_{m}=0}^{\infty}\lambda_{r_{1},...,r_{m}}.
\end{aligned}
\]

We notice that $\lambda_{r_{1},...,r_{m}}=0$ if one of $r_{1},...,r_{m}$ is $\geq n$. 

Now suppose $(r_{1},...,r_{m})\leq(s_{1},...,s_{m})$, we know that $I_{r_{1},...,r_{m}}\subseteq I_{s_{1},...,s_{m}}$ by the definition. Hence $\lambda_{r_{1},...,r_{m}}\geq\lambda_{s_{1},...,s_{m}}$.
\end{proof}

\begin{lemma}\label{boundary}
Given a $(x_{0},...,x_{m})$-primary ideal $I\subset k[x_{0},...,x_{m}]$ of codimension $n$ with m-dimensional partition $\lambda$:

(1) Fix $(r_1,...,r_m)$. Then 
\[
x_{0}^{\lambda_{r_1,...,r_m}}\prod_{i=1}^{m}x_i^{r_i}+\sum_{i=1}^{m}x_{i}P_{i}\prod_{j=1}^{i}x_j^{r_j}
\]
belongs to $I$ for some polynomials $P_i\in k[x_{0},x_{i},x_{i+1},...,x_{m}]$.

(2) If $\lambda_{r_{1},0,...,0}=0$, then $x_{1}^{r_1}\in I$.

(3) If $\lambda_{r_{1},r_{2},...,r_m}=0$, then
\[
\prod_{i=1}^{m}x_i^{r_i}+\sum_{i=1}^{m-1}x_{i}P_{i}\prod_{j=1}^{i}x_j^{r_j}
\]
belongs to $I$ for some polynomials $P_i\in k[x_{0},x_{i},x_{i+1},...,x_{m}]$.
\end{lemma}

\begin{proof}
(1) This follows from the definition of $\lambda_{r_1,...,r_m}$. 

(2) Since $\frac{(I:x_1^{r_1})}{(x_1,...,x_m)}=x_0^{\lambda_{r_{1},0,...,0}}=(1)\subset k[x_0]$ and $x_i^{n}\in I$ for $0\leq i\leq m$, we deduce that $(I:x_1^{r_1})=(1)$ by the Hilbert's Nullstellensatz. Hence $x_{1}^{r_1}\in I$ 

(3) Since $I_{r_{1},r_{2},...,r_{m}}:=(I_{r_{1},r_{2},...,r_{m-1}}:x_{m}^{r_{m}})/(x_{m})=(x_{0}^{\lambda_{r_{1},...,r_{m}}})=(1)$ and $x_m^{n}\in I$,  we deduce that $(I_{r_{1},r_{2},...,r_{m-1}}:x_{m}^{r_{m}})=(1)$ by the Hilbert's Nullstellensatz.  Then the statement follows from the definition of $I_{r_{1},r_{2},...,r_{m-1}}$.
\end{proof}

\begin{defn}
Given an $m$-dimensional partition $\lambda$ of $n$, an $m$-tuple $(r_{1},...,r_{m})\in\mathbb{N}^{m}$ is called a \emph{corner index} if for each $1\leq i\leq m$, either $r_{i}=0$ or $\lambda_{r_{1},...,r_{i}-1,...,r_{m}}>\lambda_{r_{1},...,r_{m}}$. We denote the lexicographic order on $(r_1,...,r_m)$ by $\lhd$, i.e. $(s_1,...,s_m)\lhd(r_1,...,r_m)$ if there is $1\leq i\leq m$ such that $s_k=r_k$ for $1\leq k<i$ and $s_i<r_i$.
\end{defn}

\begin{prop}\label{generators}
Given a $(x_{0},...,x_{m})$-primary ideal $I\subset k[x_{0},...,x_{m}]$ of codimension $n$, if its associated $m$-dimensional paritition of $n$ is $\lambda$, then $I$ is generated by the set of following polynomials:
\[
F_{r_{1},...,r_{m}}:=x_{0}^{\lambda_{r_{1},...,r_{m}}}x_{1}^{r_{1}}...x_{m}^{r_{m}}+\sum_{(s_{1},...,s_{m})\in J}\sum_{t=0}^{\lambda_{s_{1},...,s_{m}}-1}a_{t,s_{1},...,s_{m}}^{r_{1},...,r_{m}}x_{0}^{t}x_{1}^{s_{1}}...x_{m}^{s_{m}}
\]
for some $a_{t,s_{1},...,s_{m}}^{r_{1},...,r_{m}}\in k$, where $(r_{1},...,r_{m})$ is a corner index and 
\[
J=\{(s_{1},...,s_{m})\in\mathbb{N}^{m}|(s_1,...,s_m)\rhd(r_1,...,r_m)\}.
\]
\end{prop}

\begin{proof}
Such kinds of polynomials are contained in $I$ because of Lemma \ref{boundary} by inductions from the larger tuples to the smaller tuples. Notice that $F_{r_1,...,r_m}$ can be defined for a general $(r_1,...,r_m)$. There is no $x_{0}^{s_{0}}x_{1}^{s_{1}}...x_{m}^{s_{m}}$ term such that $s_{0}\geq\lambda_{s_{1},...,s_{m}}$ since those terms can be canceled out by $F_{r_1,...,r_m}$ with larger tuples. The polynomials $F_{r_1,...,r_m}$ can be generated by $F_{r_1,...,r_m}$ with corner indexes.

Let $I'$ be the ideal in $k[x_{0},...,x_{m}]$ generated by the polynomials $F_{r_1,...,r_m}$ with corner indexes. Then $\dim k[x_{0},...,x_{m}]/I'\geq\dim k[x_{0},...,x_{m}]/I=n$ since $I'\subseteq I$. But $k[x_{0},...,x_{m}]/I'$ is generated as a $k$-linear space by the monomials
\[
x_{0}^{t_{r_{1},...,r_{m}}}x_{1}^{r_{1}}...x_{m}^{r_{m}},\ 0\leq t_{r_{1},...,r_{m}}< \lambda_{r_{1},...,r_{m}}, \lambda_{r_{1},...,r_{m}}\neq 0.
\]
Hence $\dim k[x_{0},...,x_{m}]/I'\leq\sum\lambda_{r_{1},...,r_{m}}=n$ by Proposition \ref{partition}, which implies that $I'=I$.
\end{proof}

\begin{remark}\label{unique}
We notice that the coefficients $a_{t,s_{1},...,s_{m}}^{r_{1},...,r_{m}}$ cannot be arbitrarily chosen. Otherwise, it is possible that $\dim k[x_{0},...,x_{m}]/I'<n$. On the other hand, if we are given two $(x_{0},...,x_{m})$-primary ideals $I_{1}, I_{2}\subset k[x_{0},...,x_{m}]$ of codimension $n$ with the same partition $\lambda$, then $I_{1}=I_{2}$ if and only if the corresponding coefficients $a_{t,s_{1},...,s_{m}}^{r_{1},...,r_{m}}$ are the same for the two ideals.
\end{remark}

Although the set of generators in Proposition \ref{generators} is minimal in some sense, it is hard to write out the relations among those $a_{t,s_{1},...,s_{m}}^{r_{1},...,r_{m}}$. Hence we will include those non-corner $F_{r_1,...,r_m}$ and form another set of generators for $I$, which contains more elements but is easier to figure out the relations among coefficients.

There is a one-to-one correspondence between $m$-dimensional partitions of $n$ and certain sets of monomials in $x_0,...,x_m$ if we regard the partition as a $(m+1)$-dimensional object. Namely, the partition $\lambda=\{\lambda_{i_1,...,i_m}\}$ corresponds to the set  
\[\mathcal{O}_{\lambda}=\{x_{0}^{j}x_{1}^{i_1}...x_{m}^{i_m}|\ i_1\geq 0,...,i_m\geq 0,\ \lambda_{i_1,...,i_m}\neq 0,\ 0\leq j\leq \lambda_{i_1,...,i_m}-1\ \}.\]
We define the border $\partial\mathcal{O}_{\lambda}$ of $\mathcal{O}_{\lambda}$ by
\[
\partial\mathcal{O}_{\lambda}=\bigcup_{i=0}^{m}x_{i}\mathcal{O}_{\lambda}\backslash\mathcal{O}_{\lambda}.
\]
We want to form a new set of generators for a $(x_{0},...,x_{m})$-primary ideal $I\subset k[x_{0},...,x_{m}]$ such that each polynomial is ``leaded" by the monomials in $\partial\mathcal{O}_{\lambda}$. Let $M$ be an $(m+1)$-tuple $(r_0,...,r_m)$. We denote the monomial $x_{0}^{r_0}...x_{m}^{r_m}$ by $x^M$. 

\begin{prop}\label{newgen}
Given a $(x_{0},...,x_{m})$-primary ideal $I\subset k[x_{0},...,x_{m}]$ of codimension $n$, if its associated $m$-dimensional paritition of $n$ is $\lambda$, then $I$ is generated by the set of following polynomials:
\[
F_{M}:=x^M-\sum_{(s_{1},...,s_{m})\in J}\sum_{t=0}^{\lambda_{s_{1},...,s_{m}}-1}a_{t,s_{1},...,s_{m}}^{M}x_{0}^{t}x_{1}^{s_{1}}...x_{m}^{s_{m}}
\]
for some $a_{t,s_{1},...,s_{m}}^{M}\in k$, where $M$ is an $(m+1)$-tuple $(r_0,...,r_m)$ such that $x^M\in\partial\mathcal{O}_{\lambda}$, $\tilde{M}$ is the $m$-tuple $(r_1,...,r_m)$ which omits the first coordinate of $M$, and 
\[
J=\{(s_{1},...,s_{m})\in\mathbb{N}^{m}|(s_1,...,s_m)\rhd(r_1,...,r_m)\}.
\]

\begin{proof}
Since this set of generators contains the set of generators in Proposition \ref{generators}, it generates $I$. The coefficients of the extra polynomials here are determined by the coefficients of the polynomials in Proposition \ref{generators}. We also notice that these generators have the same uniqueness property as in Remark \ref{unique}.
\end{proof}
\end{prop}

Now we want to apply some techniques in the theory of border bases to determine the relations among those coefficients.

\begin{defn}
Let $\lambda$ be an $m$-dimensional partition of $n$. Given $r\in\{0,...,m\}$. Let $\mathcal{O}_{\lambda}=\{t_1,...,t_n\}$ be the corresponding set of monomials. Suppose $\partial\mathcal{O}_{\lambda}=\{b_1,...,b_\nu\}$. Suppose the generators in \ref{newgen} are $F_j=b_j-\sum_{i=1}^{n}\alpha_{i}^{j}t_i$. We define the $n$ by $n$ $r$-th \emph{formal multiplication matrix} $T_r=(\xi_{kl}^r)$ of $\mathcal{O}_{\lambda}$ by 
\[
\xi_{kl}^r=\begin{cases}
\delta_{ki}, \ \text{if }x_{r}t_l=t_i\\
\alpha_{k}^{j},\ \text{if }x_{r}t_l=b_j
\end{cases},
\]
where $\delta_{ki}=1$ if $k=i$ and $\delta_{ki}=0$ otherwise.
\end{defn}

These matrices describe the multplication operations on the linear space $V$, the basis of which consists of elements in $\mathcal{O}_{\lambda}$. For example, the formal multiplication matrix $T_0$ is the matrix for multiplicaton by $x_{0}$. Let $t_l$ be a monomial in $\mathcal{O}_{\lambda}$. If $x_{0}t_l\in\mathcal{O}_\lambda$, then $T_{0}(t_l)=x_{0}t_l$. If $x_{0}t_l=b_j$, then $T_{0}(t_l)=\sum_{i=1}^{n}\alpha_{i}^{j}t_i$.
 
The following theorems are proved in \cite{KR05} with slightly different presentations.
 
\begin{theorem}\label{com}\cite[Theorem 6.4.30.]{KR05}
Let $\lambda$ be an $m$-dimensional partition of $n$. Let $\mathcal{O}_{\lambda}=\{t_1,...,t_n\}$ be the corresponding set of monomials. Suppose $\partial\mathcal{O}_{\lambda}=\{b_1,...,b_\nu\}$. Define the polynomials $F_j=b_j-\sum_{i=1}^{n}\alpha_{i}^{j}t_i$ for $j=1,...,\nu$. Let $I$ be the ideal in $\mathbb{C}[x_0,...,x_m]$ generated by the polynomials $F_j$. Then the following conditions are equivalent:

a) The dimension of $\bbc[x_0,...,x_m]/I$ is $n$.

b) The formal multiplication matrices of $\mathcal{O}_{\lambda}$ are pairwise commuting.

In that case the formal multiplication matrices represent the multiplication endormorphism of $\bbc[x_0,...,x_m]/I$ with respect to the basis $\overline{t_1},...,\overline{t_n}$.
\end{theorem}

Since the matrices are explicit, we can write down all the relations of the coefficients of $F_j$ such that the formal multiplication matrices are pairwise commuting.

\begin{theorem}\label{hardrel}\cite[Proposition 6.4.32.]{KR05}\footnote{We correct some typos and add a missing set of equations.}
Using the same notations in theorem \ref{com}, we define a map for $r\in\{0,...,m\}$ and $i\in\{1,...,n\}$
\[
\rho_{r}(i)=\begin{cases}
j,\ \text{if } x_{r}t_i=t_j\in\mathcal{O}_{\lambda}\\
k,\ \text{if } x_{r}t_i=b_k\in\partial\mathcal{O}_{\lambda}
\end{cases}.
\]
Then the formal multiplication matrices are pairwise commuting if and only if the following equations are satisfied for $i,p\in\{1,...,n\}$ and $0\leq r<s\leq m$:
\[
(1) \sum_{\{d|x_{r}t_{d}\in\mathcal{O}_{\lambda}\}}\delta_{p\rho_{r}(d)}\alpha_{d}^{k}+\sum_{\{d|x_{r}t_{d}\in\partial\mathcal{O}_{\lambda}\}}\alpha_{p}^{\rho_{r}(d)}\alpha_{d}^{k}=\alpha_{p}^{l}
\text{ if } x_{r}t_{i}=t_{j},\ x_{s}t_{i}=b_{k},\ x_{r}b_{k}=b_l
\]
\[
(2) \sum_{\{d|x_{s}t_{d}\in\mathcal{O}_{\lambda}\}}\delta_{p\rho_{s}(d)}\alpha_{d}^{k}+\sum_{\{d|x_{s}t_{d}\in\partial\mathcal{O}_{\lambda}\}}\alpha_{p}^{\rho_{s}(d)}\alpha_{d}^{k}=\alpha_{p}^{l}
\text{ if } x_{s}t_{i}=t_{j},\ x_{r}t_{i}=b_{k},\ x_{s}b_{k}=b_l
\]
\[\begin{aligned}
(3) \sum_{\{d|x_{r}t_{d}\in\mathcal{O}_{\lambda}\}}\delta_{p\rho_{r}(d)}\alpha_{d}^{k}+\sum_{\{d|x_{r}t_{d}\in\partial\mathcal{O}_{\lambda}\}}\alpha_{p}^{\rho_{r}(d)}\alpha_{d}^{k}=
\sum_{\{d|x_{s}t_{d}\in\mathcal{O}_{\lambda}\}}\delta_{p\rho_{s}(d)}\alpha_{d}^{j}+\sum_{\{d|x_{s}t_{d}\in\partial\mathcal{O}_{\lambda}\}}\alpha_{p}^{\rho_{s}(d)}\alpha_{d}^{j}\\
\text{if } x_{r}t_{i}=b_{j},\ x_{s}t_{i}=b_{k}
\end{aligned}\]
\end{theorem}

\begin{proof}[Proof of Theorem \ref{decomp}]
By Proposition \ref{newgen}, Theorem \ref{com} and Theorem \ref{hardrel}, the affine variety $V_\lambda$ is defined by the variables/coefficients $\alpha_*^*$ quotient by the relations $(1),(2),(3)$. The decomposition in the Grothendieck of varieties follows from \cite[Lemma A.4]{LL20}. 
\end{proof}

We give some examples using the notation in Proposition \ref{newgen}.

\begin{example}
\begin{tabular}{|c|c|c|c|c|}
\hline
$3$ & $2$ & $1$ & $1$ & $1$\\
\hline
\end{tabular}
\[
m=1,\ \lambda_0=3,\ \lambda_1=2,\ \lambda_2=1,\ \lambda_3=1,\ \lambda_4=1.
\]
Ideals $I\subset k[z,x]$ associated with this partition have the form
\[
I=(z^3-a^{30}_{01}x-a^{30}_{11}xz-a^{30}_{02}x^2-a^{30}_{03}x^3-a^{40}_{04}x^4,\ z^2x-a^{21}_{02}x^2-a^{21}_{03}x^3-a^{21}_{04}x^4,
\]
\[
zx^2-a^{12}_{03}x^3-a^{12}_{04}x^4,\ zx^3-a^{13}_{04}x^4,\ x^5)
\]
The relations between the coefficients are
\[
a^{30}_{03}+a^{30}_{11}a^{12}_{04}=a^{21}_{02}a^{12}_{04}+a^{21}_{03}a^{13}_{04},\ a^{30}_{02}+a^{30}_{11}a^{12}_{03}=a^{21}_{02}a^{12}_{03},
\]
\[
a^{30}_{01}=0,\ a^{21}_{03}=a^{12}_{03}a^{13}_{04},\ a^{21}_{02}=0,\ a^{12}_{03}=a^{13}_{04}.
\]
Hence $V_\lambda$ is an affine space $\bba^5$ with free variables $a^{30}_{04},a^{30}_{11},a^{21}_{04},a^{12}_{04},a^{13}_{04}$.
\end{example}

\begin{remark}\label{order}
When $m=1$, one can define a total order on the variables $a^*_*$. Then one observes that all the variables are expressed by larger variables. Hence the stratum $V_\lambda$ is always an affine space with the expected dimension.
\end{remark}

\begin{example}
\begin{tabular}{|c|c|c|}
  \cline{1-1}
  $1$  &\multicolumn{1}{|c}{}  \\
   \hline
   $1$ & $1$ & $1$  \\
  \hline
\end{tabular}
\[
m=2, \lambda_{0,0}=1,\ \lambda_{1,0}=1,\ \lambda_{2,0}=1,\ \lambda_{0,1}=1.
\]
Ideals $I\subset k[z,x,y],$ associated with this partition have the form
\[
I=(z-a^{100}_{010}x-a^{100}_{020}x^2-a^{100}_{001}y,\ zx-a^{110}_{020}x^2,\, zy-a^{101}_{010}x-a^{101}_{020}x^2,\, y^2-a^{002}_{010}x-a^{002}_{020}x^2,\ xy-a^{011}_{020}x^2,\ zx^2,\ x^3)
\]
The relations between the coefficients are
\[
a^{101}_{010}=0,\ a^{100}_{010}+a^{100}_{001}a^{011}_{020}=a^{110}_{020},\ a^{101}_{010}a^{011}_{020}=a^{110}_{020}a^{002}_{010}
\]
\[
a^{100}_{001}a^{002}_{020}+a^{100}_{010}a^{011}_{020}=a^{101}_{020},\ a^{100}_{001}a^{002}_{010}=a^{101}_{010},\ a^{002}_{010}=0
\]
Hence $V_\lambda$ is an affine space $\bba^5$ with free varibles $a^{100}_{010},a^{100}_{020},a^{100}_{001},a^{002}_{020},a^{011}_{020}$.
\end{example}

However, $V_\lambda$ is not an affine space in general. The reason for not being affine is that only quadratic terms are left in some relations or some variabes are expressed in different ways. One can carry out the following calculations by hand or by programme.

\begin{example}\label{minimal}

\begin{tabular}{|c|c|}
\cline{1-1}
$1$ &\multicolumn{1}{|c}{}  \\
\hline
$2$ & $1$  \\
\hline
\end{tabular}

Only one quadratic term is left in the relation from $T_yT_z(t_{100})=T_zT_y(t_{100})$:
\[
a^{002}_{010}a^{200}_{001}=0
\]
All the other variables are free variables. Hence it is not an affine space. We deduce that  $[V_\lambda]=2L^3-L^2$ in the Grothendieck group of varieties by calculations. 
\end{example}

\begin{example}

\begin{tabular}{|c|c|}
\cline{1-1}
$1$ &\multicolumn{1}{|c}{}  \\
\cline{1-1}
$1$ &\multicolumn{1}{|c}{}  \\
\hline
$2$ & $1$  \\
\hline
\end{tabular}

Excluding free variables and the variables determined by other variables, we have the relation 
\[
a_{010}^{003}((a_{002}^{101})^2-a_{002}^{200})=0.
\]
We deduce that  $[V_\lambda]=2L^4-L^3$ in the Grothendieck group of varieties by calculations. 
\end{example}

\begin{example}

\begin{tabular}{|c|c|c|}
\cline{1-1}
$1$ &\multicolumn{1}{|c}{}  \\
\hline
$2$ & $1$ & $1$ \\
\hline
\end{tabular}

Excluding free variables and the variables determined by other variables, we have the relation 
\[
a_{001}^{200}((a_{020}^{011})^2-a_{020}^{002})=0.
\]
We deduce that $[V_\lambda]=2L^5-L^4$ in the Grothendieck group of varieties by calculations. 
\end{example}

\begin{example}

\begin{tabular}{|c|c|}
\cline{1-1}
$1$ &\multicolumn{1}{|c}{}  \\
\hline
$3$ & $1$  \\
\hline
\end{tabular}

Excluding free variables and the variables determined by other variables, we have the relations
\[
a_{010}^{002}a_{001}^{300}=0
\]
We deduce that $[V_\lambda]=2L^3-L^2$ in the Grothendieck group of varieties by calculations. 
\end{example}

\begin{proof}[Proof of Theorem \ref{genseries1}]
We use computer programme to generate all the 2-dimensional partitions of $n$ and calculate the dimension of $V_\lambda$ for $n\leq 5$. This is done by the observation that all the strata are affine for $n\leq 5$ except for the above examples. Then we counting the number of points over $\bbf_2$ for those affine strata. See Appendix.
\end{proof}

Via the power structure over the Grothendieck group of varieties \cite{ZLH06}, the relation between the Hilbert schemes of points and the punctual Hilbert schemes of points can be expressed as follows.

\begin{thm}\cite[Theorem 1]{ZLH06}
For a smooth quasi-projective variety $X$ of dimension $d$, the following identity holds in $K_{0}({\rm Var}_\bbc)[[T]]$:
\[
\mathbb{H}_X(T)=(\mathbb{H}_{\bba^d,0}(T))^{[X]},
\]
where
\[
\mathbb{H}_X(T):=1+\sum_{n=1}^{\infty}[{\rm Hilb}^{n}(X)]T^n,\ \ \mathbb{H}_{\bba^d,0}(T):=1+\sum_{n=1}^{\infty}[{\rm Hilb}^{n}_{0}(\bba^d)]T^n
\]
\end{thm}

\begin{proof}[Proof of Theorem \ref{genseries2}]
Denote $K_{0}({\rm Var}_\bbc)$ by $R$. Any series $A(T)\in 1+T\cdot R[[T]]$ can be uniquely written as a product of the form $\prod_{i=1}^{\infty}(1-T^i)^{-a_i}$ with $a_i\in R$, and 
\[
(A(T))^m=\prod_{i=1}^{\infty}(1-T^{i})^{-a_{i}m}
\]
for $m\in R$.
When $A(T)=\mathbb{H}_{\bba^3,0}(T)$, $a_i$ is a polynomial in $L$ for $i\leq 5$ by Theorem \ref{genseries1}, and 
\[
\mathbb{H}_{\bba^3}(T)=(\mathbb{H}_{\bba^3,0}(T))^{[\bba^3]}.
\]
Hence if we can calculate $(1-T^i)^{kL}$ for $k\in\bbz$, we can calculate the coefficients of $T^n$ in the left hand side for $n\leq 5$. But $(1-T)^{-[M]}=\sum_{n=0}^{\infty}[S^{n}M]T^n$ (see \cite{ZLH06}), where $S^{n}M=M^{n}/S_n$ is the $n$th symmetric power of $M$, and $[S^{n}(\bba^m)]=[\bba^{nm}]$ in $K_{0}({\rm Var}_\bbc)$. So we can carry out the calculation by programme.
\end{proof}

\section{Appendix}

\begin{lemma}\label{line}
All the strata as  
\begin{tabular}{|c|c|c|c|}
\hline
$1$ & $1$ & $\cdots$ & $1$\\
\hline
\end{tabular}
are affine spaces.
\end{lemma}

\begin{proof}
All the relations from $T_xT_y=T_yT_x$ and $T_xT_z=T_zT_x$ are given by one variable being equal to another variable. We claim that they imply $T_yT_z=T_zT_y$. This is done inductively by checking $T_yT_z(t_{0x0})=T_zT_y(t_{0x0})$ from larger $x$ to smaller $x$. When $x=n-1,n-2$, $T_yT_z(t_{0x0})=T_zT_y(t_{0x0})=0$. Suppose $T_yT_z(t_{0x0})=T_zT_y(t_{0x0})$ for $n-2\geq x> k$. Then
\[
T_yT_z(t_{0k0})=T_yT_zD_x(t_{0,k+1,0})=T_yD_xT_z(t_{0,k+1,0})
\]
\[
=T_zD_xT_y(t_{0,k+1,0})=T_zT_yD_x(t_{0,k+1,0})=T_zT_y(t_{0k0}),
\]
where $D_x$ is the endormorphism sending $t_{0x0}$ to $t_{0,x-1,0}$ (not the inverse of $T_x$ since $T_x$ is not invertible) and we used $T_xT_y=T_yT_x$ implicitly. 
\end{proof}

\begin{lemma}
All the strata as  
\begin{tabular}{|c|c|c|c|}
\hline
$a$ & $1$ & $\cdots$ & $1$\\
\hline
\end{tabular}
are affine spaces, $a>1$.
\end{lemma}

\begin{proof}
We first notice that $T_xT_z=T_zT_x,T_yT_z=T_zT_y$ and $T_xT_y=T_yT_x$ is equivalent to $T_xT_z=T_zT_x,T_yT_z=T_zT_y$ and $T_xT_y(t_{0ij})=T_yT_x(t_{0ij})$ for all $i,j$. Then as in Remark \ref{order} and Lemma \ref{line}, it suffices to check $T_zT_y(t_{a-1,0,0})=T_yT_z(t_{a-1,0,0})$ is implied by other relations. Notice that $T_y$ send the last term of $T_z(t_{a-1,0,0})$ to $0$. Hence $D_x$ is well-defined during the calculation below:
\[
T_yT_z(t_{a-1,0,0})=T_yD_xT_xT_z(t_{a-1,0,0})=T_yD_xT_zT_x(t_{a-1,0,0})=T_zD_xT_yT_x(t_{a-1,0,0})
\]
\[
=T_zD_xT_yT_z^{a-1}T_x(t_{000})=T_zD_xT_yT_z^{a-1}(t_{010})=T_zD_xT_z^{a-1}(t_{011})=T_z^{a}(t_{001})=T_zT_y(t_{a-1,0,0})
\]
\end{proof}

\begin{lemma}
All the strata as  
\begin{tabular}{|c|c|c|c|}
\cline{1-1}
$1$ &\multicolumn{1}{|c}{}  \\
\hline
$1$ & $1$ & $\cdots$ & $1$ \\
\hline
\end{tabular}
are affine spaces.
\end{lemma}

\begin{proof}
We need to check that $T_zT_y(t_{001})=T_yT_z(t_{001})$ and $T_zT_y(t_{000})=T_yT_z(t_{000})$ are implied by other relations or given by affine spaces. Notice that $T_y$ send the last term of $T_z(t_{001})$ to $0$. Hence $D_x$ is well-defined during the calculation below:
\[
T_yT_z(t_{001})=T_yD_xT_xT_z(t_{001})=T_yD_xT_zT_x(t_{001})=T_zD_xT_yT_x(t_{001})=T_zT_y(t_{001})
\]

Notice that $a_{0x0}^{101}$ is expressed in two quadrics in $T_zT_y(t_{000})=T_yT_z(t_{000})$ and $T_zT_x(t_{001})=T_xT_z(t_{001})$ except when $x=n-2$. Hence we only need to show that the two expressions are the same. Let $V$ be the vector space generated by the basis vectors $t_{001},t_{0x0},0\leq x\leq n-2$. Let $W$ be the subspace generated by those vectors excluding $t_{0,n-2,0}$. In other words, we show that $T_zT_y(t_{000})|_W=T_yT_z(t_{000})|_W$ is implied by $T_zT_x(t_{001})=T_xT_z(t_{001})$ and other relations:
\[
T_zT_y(t_{000})|_W=T_zD_xT_x(t_{001})|_W=T_zD_xT_y(t_{010})|_W=T_yD_xT_z(t_{010})|_W
\]
\[
=T_yD_xT_xT_z(t_{000})|_W=T_yT_z(t_{000})|_W
\]
\end{proof}

\begin{lemma}
All the strata as  
\begin{tabular}{|c|c|c|c|}
\cline{1-1}
$1$ &\multicolumn{1}{|c}{}  \\
\cline{1-1}
$\vdots$ &\multicolumn{1}{|c}{}  \\
\hline
$1$ & $1$ & $\cdots$ & $1$ \\
\hline
\end{tabular}
are affine spaces.
\end{lemma}

\begin{proof}
Similar as above.
\end{proof}

\begin{lemma}
The strata  
\begin{tabular}{|c|c|}
\hline
$1$ & $1$ \\
\hline
$1$ & $1$\\
\hline
\end{tabular}\hspace{0.5cm}
\begin{tabular}{|c|c|c|}
\cline{1-2}
$1$ & $1$ & \multicolumn{1}{|c}{}\\
\hline
$1$ & $1$ & $1$\\
\hline
\end{tabular}
are affine spaces.
\end{lemma}

\begin{proof}
We prove the first case. The second case is similar. The variable $a^{101}_{010}$ is expressed in two quadrics in $T_zT_y(t_{000})=T_yT_z(t_{000})$ and $T_zT_y(t_{001})=T_yT_z(t_{001})$, and we will show the second is implied by the first and some other relations:
\[
T_yT_z(t_{001})=T_yT_zT_y(t_{000})=T_yT_yT_z(t_{000})=T_yT_yD_xT_xT_z(t_{000})=T_yT_yD_xT_zT_x(t_{000})
\]
\[
=T_yT_yD_xT_z(t_{010})=a^{110}_{011}a^{002}_{010}t_{011}=T_zT_y(t_{001})
\]
\end{proof}

\begin{lemma}
The strata  
\begin{tabular}{|c|c|}
\hline
$1$ & $1$ \\
\hline
$2$ & $1$\\
\hline
\end{tabular}\hspace{0.5cm}
\begin{tabular}{|c|c|}
\cline{1-1}
 $1$ & \multicolumn{1}{|c}{}\\
\hline
 $2$ & $2$\\
\hline
\end{tabular}
are affine spaces.
\end{lemma}

\begin{proof}
We prove the first case. The second case is similar. The only problem we may have is $a^{200}_{001}a^{002}_{010}=0$ from $T_yT_z(t_{100})=T_zT_y(t_{100})$. But $a^{200}_{001}=0$ from $T_zT_x(t_{100})=T_xT_z(t_{100})$, so the stratus is affine.
\end{proof}

\begin{lemma}
The stratum
\begin{tabular}{|c|c|c|}
\hline
$2$ & $2$ & $1$\\
\hline
\end{tabular}
is affine.
\end{lemma}

\begin{proof}
We only need to check that $T_zT_y(t_{100})=T_yT_z(t_{100})$ is implied by other relations:
\[
T_yT_z(t_{100})=T_yD_xT_z(t_{110})=a^{210}_{020}a^{011}_{020}t_{020}=a^{210}_{020}a^{101}_{110}t_{020}=T_zT_y(t_{100}),
\]
where we implicitly use $a^{111}_{020}=0$ from $T_zT_y(t_{010})=T_yT_z(t_{010})$.
\end{proof}

\begin{remark}
The above lemmas and the examples in Sectoin 2 cover all the cases that need to be checked for $n\leq 5$. For the cases not being covered, either all the relations are linear or there is no variable being expressed twice involving quadrics. Hence those strata are affine spaces. We believe there should be a more uniform and geometric way to describe this phenomenon.
\end{remark}

\end{document}